\documentclass{amsart}

\usepackage{amsthm,amsmath,amssymb,amsfonts,dsfont}
\usepackage[colorlinks, linkcolor=blue, citecolor=blue]{hyperref}
\usepackage[indentafter]{titlesec}
\titleformat{name=\section}{}{\thetitle.}{0.8em}{\centering\scshape}
\usepackage{mathrsfs}
\usepackage{bm}
\usepackage{titlesec}
\numberwithin{equation}{section}
\titleformat{\subsection}[runin]
  {\normalfont\bfseries}{\thesubsection}{1em}{}
\newtheorem{definition}{\textbf{Definition}}[section]
\newtheorem{theorem}[definition]{\textbf{Theorem}}
\newtheorem{corollary}[definition]{\textbf{Corollary}}
\newtheorem{lemma}[definition]{\textbf{Lemma}}

\newtheorem{letterthm}{Theorem}

\usepackage{thmtools}
\swapnumbers
\declaretheoremstyle[bodyfont=\normalfont]{normalbody}

\declaretheorem[numberlike=hyperstate,style=normalbody, name=Weakly mixing bimodules and property (T)]{prop (T)}
\declaretheorem[numberlike=hyperstate,style=normalbody, name=Mixing bimodules and Haagerup property]{prop (H)}

\declaretheorem[numbered=no,name=Theorem A]{Theorem A}
\declaretheorem[numbered=no,name=Theorem B]{Theorem B}
\declaretheorem[numbered=no,name=Theorem C]{Theorem C}

\def\d{\mathrm{d}}

\def\CS{\mathcal{S}}
\def\CB{\mathcal{B}}
\def\CP{\mathcal{P}}
\def\CA{\mathcal{A}}
\def\har{\mathrm{Har}}

\def\N {\mathbb{N}}

\def\C {\mathbb{C}}
\def\UCP {\mathrm{UCP}}

\title[Hypertrace and entropy gap characterizations of property (T)]{Hypertrace and entropy gap characterizations of property (T) for $\mathrm{II}_1$ factors}
\author{Shuoxing Zhou}
\address{\'Ecole Normale Sup\'erieure\\ D\'epartement de math\'ematiques et applications\\ 45 rue d'Ulm\\ 75230 Paris Cedex 05\\ FRANCE}
\email{shuoxing.zhou@ens.psl.eu}

\begin{document}

\maketitle
\begin{abstract}
We establish a hypertrace characterization of property (T) for $\mathrm{II}_1$ factors: Given a $\mathrm{II}_1$ factors $M$, $M$ does not have property (T) if and only if there exists a von Neumann algebra $\mathcal{A}$ with $M\subset \mathcal{A}$ such that $\mathcal{A}$ admits a $M$-hypertrace but no normal hypertrace. For $M$ without property (T), such an inclusion $M\subset \mathcal{A}$ also admits almost vanishing Furstenberg entropy. With the same construction of $M\subset \mathcal{A}$, we also establish similar characterizations of Haagerup property for $\mathrm{II}_1$ factors.
\end{abstract}

\section{Introduction}

In ergodic group theory, an important way to study the properties of a countable discrete group $\Gamma$ is by considering different nonsingular actions $\Gamma \curvearrowright (X,\nu_X)$. A classical example is to study amenability of $\Gamma$ by the left translation action $\Gamma \curvearrowright \Gamma$. For a nonsingular action $\Gamma \curvearrowright (X,\nu_X)$, we have $L(\Gamma)\subset L(\Gamma \curvearrowright X)$. Therefore, for a tracial von Neumann algebra $(M,\tau)$, the noncommutative analogue of $\Gamma \curvearrowright X$ is naturally the inclusion $M\subset \CA$ for different von Neumann algebras $\CA$. A classical example for the application of $M\subset\CA$ is to study amenability of $(M,\tau)$ through the inclusion $M\subset B(L^2(M))$.

Just like amenability, property (T) of $\Gamma$ can also be studied through the group actions $\Gamma \curvearrowright (X,\nu_X)$. For a nonsingular action $\Gamma \curvearrowright (X,\nu_X)$ and a measure $\mu\in\mathrm{Prob}(\Gamma)$, the \textbf{Furstenberg entropy} \cite{Fu63} of $(X,\nu_X)$ with respect to $\mu$ is defined to be 
$$h_\mu(X,\nu_X)=-\int_{\Gamma\times X}\log\left(\frac{\d \gamma^{-1}\nu_X}{\d \nu_X}(x)\right)\d \mu(\gamma)\d \nu_X(x).$$
A nonsingular action $\Gamma \curvearrowright (X,\nu_X)$ is said to be properly nonsingular if $\nu_X$ is not equivalent to a $\Gamma$-invariant probability measure. It is proved in \cite{Ne03} and \cite{BHT14} that $\Gamma$ has property (T) if and only if it has entropy gap, that is, for any generating measure $\mu\in\mathrm{Prob}(\Gamma)$, there exists $\epsilon=\epsilon(\mu)>0$ such that for any ergodic, properly nonsingular action $\Gamma \curvearrowright (X,\nu_X)$, one has $h_\mu(X,\nu_X)>\epsilon$.

Recently, Das-Peterson \cite{DP20} extended the entropy theory to the noncommutative setting $M\subset \CA$, which we refer to \ref{hyperstate} and \ref{NC entropy} for a brief introduction of the definition. They also proved that any $\mathrm{II}_1$ factor with property (T) has entropy gap (Definition \ref{entropy gap}), while the converse is left open. Then Amine Marrakchi asked the following question: Does entropy gap characterize property (T) for $\mathrm{II}_1$ factors? The goal of this paper is to answer this question. 

Recently, with the tool of Gaussian functor, Arano-Isono-Marrakchi \cite[Corollary 7.5]{AIM21} established a series of new characterizations of property (T) for locally compact group, including the entropy gap characterization: For a locally compact group $G$, the following are equivalent:
\begin{itemize}
\item [(i)] $G$ does not have property (T);
\item [(ii)] $G$ admits a nonsingular action which has an invariant mean but no invariant probability measure;
\item[(iii)] $G$  admits a nonsingular action which has almost vanishing entropy but no invariant probability measure.
\end{itemize}

Moreover, Arano-Isono-Marrakchi \cite[Corollary 7.6]{AIM21} also established a series of similarly characterizations of Haagerup property for locally compact group: For a locally compact group $G$, the following are equivalent:
\begin{itemize}
\item [(i)] $G$ has the Haagerup property;
\item [(ii)] $G$ admits a nonsingular action $\sigma: G \curvearrowright X$ of zero-type (i.e., the Koopman representation $\pi_\sigma: G \curvearrowright L^2(X)$ is mixing) which has an invariant mean;
\item[(iii)] $G$ admits a nonsingular action of zero-type with almost vanishing entropy.
\end{itemize}

Inspired by the results above, with the tool of Shlyakhtenko's $A$-valued semicircular system \cite{Sh99}, which can be considered as the noncommutative analogue of Gaussian factor, we establish a series of similar characterizations of property (T) for $\mathrm{II}_1$ factor and we answer Marrakchi's question. Before presenting the theorem, we refer to subsection \ref{def (T)} for the definition of property (T) and subsection \ref{hyperstate} and \ref{NC entropy} for the notations of entropy theory on $M\subset \CA$.

\begin{letterthm}\label{thm A}
Let $(M,\tau)$ be a separable $\mathrm{II}_1$ factor. The following conditions are equivalent.
\begin{itemize}
\item [(i)] $M$ does not have property (T);
\item [(ii)]There exists an inclusion $M\subset\CA$ such that $L^2(\CA)$ has $M$-almost central unit vectors, but no non-zero $M$-central vectors;
\item [(iii)] There exists an inclusion $M\subset\CA$ such that there exists a $(M,\tau)$-hypertrace on $\CA$, but no normal hypertrace;
\item [(iv)] There exists an inclusion $M\subset\CA$ such that there exists a conditional expectation from $\CA$ onto $M$, but no normal conditional expectation;
\item [(v)] There exists an inclusion $M\subset\CA$ without normal conditional expectation that admits almost vanishing Furstenberg entropy, that is, there exists a normal regular strongly generating hyperstate $\varphi\in\CS_\tau(B(L^2(M)))$ and a net of normal faithful hyperstates $(\varphi_i)_{i\in I}$ on $\CA$ such that $$\lim_i h_\varphi(\CA,\varphi_i)=0;$$
\item [(vi)] $M$ does not have entropy gap.
\end{itemize}
\end{letterthm}

The construction of $M\subset\CA$ in Theorem \ref{thm A} also applies to $\mathrm{II}_1$ factors with Haagerup property. Hence we have the following theorem of characterization of Haagerup property for $\mathrm{II}_1$ factors, which is is the noncommutative analogue of \cite[Corollary 7.6]{AIM21}:
\begin{letterthm}\label{thm B}
Let $(M,\tau)$ be a separable $\mathrm{II}_1$ factor. The following conditions are equivalent.
\begin{itemize}
\item [(i)] $M$ has the Haagerup property;
\item [(ii)] There exists an inclusion $M\subset\CA$ such that $L^2(\CA)$ is $M$-mixing and has $M$-almost central unit vectors;
\item [(iii)] There exists an inclusion $M\subset\CA$ such that $L^2(\CA)$ is $M$-mixing and there exists a $(M,\tau)$-hypertrace on $\CA$;
\item [(iv)] There exists an inclusion $M\subset\CA$ such that $L^2(\CA)$ is $M$-mixing and $M\subset \CA$ admits almost vanishing Furstenberg entropy.
\end{itemize}
\end{letterthm}

\noindent{\bf Conventions.} In this paper, we only consider separable von Neumann algebras and separable Hilbert bimodules. When $\CA$ is a von Neumann algebra, by an inclusion $M\subset\CA$, we mean $M$ can be embedding into $\CA$ as a von Neumann subalgebra (rather than just C$^*$-subalgebra). For any $\sigma$-finite von Neumann algebra $N$, we denote by $(N,L^2(N),J_N, L^2(N)_+)$ the standard form (see \cite{Ha75}) of $N$. Without special instructions, any inner product is linear in the first slot and anti-linear in the second one.

\section{Preliminaries}\label{Preliminaries}

\subsection{Hyperstates and u.c.p. maps.}\label{hyperstate}

Fix a tracial von Neumann algebra $(M, \tau )$. Following \cite{DP20}, for a  C$^*$-algebra $\CA$ such that $M\subset\CA$ and a state $\psi$ on $\CA$, we say that $\psi$ is a $\tau$-\textbf{hyperstate} if $\psi|_M =\tau$. And we say that a $\tau$-hyperstate $\psi$ is a \textbf{hypertrace} if for any $x\in M$, $x\psi=\psi x$ (i.e. $\psi(x\,\cdot\,)=\psi(\,\cdot\,x)$). 

We denote by $\mathcal{S}_\tau(\CA)$ the set of $\tau$-hyperstates on $\CA$. For $\psi\in\mathcal{S}_\tau(\CA)$, we naturally have $L^2(M,\tau)\subset L^2(\CA,\psi)$. For convention, we simply denote $L^2(M,\tau)$ by $L^2(M)$. Let $e_M\in B(L^2(\CA,\psi))$ be the orthogonal projection onto $L^2(M)$. The u.c.p. map $\CP_\psi:\CA\to B(L^2(M))$ is defined as 
$$\CP_\psi(T)=e_MTe_M, \ T\in \CA.$$
Following \cite[Proposition 2.1]{DP20}, $\psi\mapsto\CP_\psi$ is a bijection between hyperstates on $\CA$ and u.c.p. $M$-bimodular maps from $\CA$ to $B(L^2(M))$, whose inverse is $\CP \mapsto \langle\CP(\,\cdot\,)\hat{1},\hat{1}\rangle$. When $\CA$ is a von Neumann algebra, $\psi$ is normal if and only if $\CP_\psi$ is normal.

For $\psi\in\mathcal{S}_\tau(\CA)$ and $\varphi \in \mathcal{S}_\tau(B(L^2(M)))$, the convolution $\varphi \ast\psi\in \mathcal{S}_\tau(\CA)$ is defined to be the hyperstate associated with the $M$-bimodular u.c.p. map $\mathcal{P}_{\varphi}\circ \mathcal{P}_{\psi}$. And $\psi$ is said to be $\varphi$-stationary if $\varphi \ast\psi=\psi$.

Let $\varphi\in\mathcal{S}_\tau(B(L^2(M)))$ be a hyperstate. The set of $\mathcal{P}_\varphi$\textbf{-harmonic operators} is defined to be
$$\mathrm{Har}(\mathcal{P}_\varphi)=\mathrm{Har}(B(L^2(M)),\mathcal{P}_\varphi)=\{ T\in B(L^2(M))\mid \mathcal{P}_\varphi(T)=T\}. $$
The \textbf{noncommutative Poisson boundary} $\mathcal{B}_\varphi$ of $M$ with respect to $\varphi$ is defined to be the noncommutative Poisson boundary of the u.c.p. map $\mathcal{P}_\varphi$ as defined by Izumi \cite{Izu02}, that is, the Poisson boundary $\mathcal{B}_\varphi$ is the unique $\mathrm{C}^*$-algebra (a von Neumann algebra when $\varphi$ is normal) that is isomorphic, as an operator system, to the space of harmonic operators $\mathrm{Har}(\mathcal{P}_\varphi)$. And the isomorphism $\CP:\CB_\varphi\to \mathrm{Har}(\mathcal{P}_\varphi)$ is called the $\varphi$-\textbf{Poisson transform}. Since $M\subset\mathrm{Har}(\mathcal{P}_\varphi)$, $M$ can also be embedded into $\CB_\varphi$ as a subalgebra. We said that $\CB_\varphi$ is \textbf{trivial} if $\CB_\varphi=M$. For the inclusion 
$M\subset \CB_\varphi$, $\zeta:=\varphi\circ\CP\in\mathcal{S}_\tau(\CB_\varphi)$ is the \textbf{canonical $\varphi$-stationary hyperstate} on $\CB_\varphi$.

Following \cite[Proposition 2.8]{DP20}, for a normal hyperstate $\varphi\in\mathcal{S}_\tau(B(L^2(M)))$, there exists a sequence $\{z_n\}\subset M$ such that $\sum_{n=1}^{\infty}z^*_nz_n=1$, and $\varphi$ and $\CP_\varphi$ admit the following standard form:
$$\varphi(T)=\sum_{n=1}^{\infty}\langle T  \hat{z}_n^*,\hat{z}_n^*\rangle, \ \mathcal{P}_\varphi(T)=\sum_{n=1}^{\infty} (Jz_n^*J)T(Jz_nJ), \ T\in B(L^2(M)).$$
Following \cite[Proposition 2.5 and the unnumbered remark right after Proposition 2.8]{DP20}, $\varphi$ is said to be 
\begin{itemize}
\item \textbf{regular}, if $\sum_{n=1}^{\infty}\limits z_n^*z_n=\sum_{n=1}^{\infty}\limits z_nz_n^*=1$;
\item \textbf{strongly generating}, if the unital algebra (rather than the unital $\ast$-algebra) generated by $\{z_n\}$ is weakly dense in $M$.
\end{itemize}
Following \cite[Proposition 2.9]{DP20}, when $\varphi$ is a normal regular strongly generating hyperstate, the canonical hyperstate $\zeta=\varphi\circ\CP$ is a normal faithful hyperstate on $\CB_\varphi$.

\subsection{Entropy.}\label{NC entropy} 

Following \cite{DP20}, for a normal hyperstate $\varphi\in\mathcal{S}_\tau(B(L^2(M)))$, let $A_\varphi\in B(L^2(M))$ be the trace class operator associated to $\varphi$. The \textbf{entropy} of $\varphi$ is defined to be 
$$H(\varphi)=-\mathrm{Tr}(A_\varphi \log A_\varphi).$$
The \textbf{asymptotic entropy} of $\varphi$ is defined to be 
$$h(\varphi)=\lim_{n\to\infty}\frac{H(\varphi^{*n})}{n}.$$

For a von Neumann algebra $\CA$ such that $M\subset \CA$ and a normal faithful hyperstate $\zeta\in \mathcal{S}_\tau(\CA)$, let $\Delta_\zeta:L^2(\CA,\zeta)\to L^2(\CA,\zeta)$ be the modular operator of $(\CA,\zeta)$, i.e., $\Delta_\zeta=S_\zeta^*S_\zeta$, where $S_\zeta:L^2(\CA,\zeta)\to L^2(\CA,\zeta)$ is an unbounded anti-linear operator given by the closure of $S_0:a\xi_\zeta\mapsto a^*\xi_\zeta$ ($a\in \CA$, $\xi_\zeta\in L^2(\CA,\zeta)$ is the cyclic vector). For the details regarding Tomita-Takesaki theory, we refer to \cite[Chapter VI-IX]{TakII}. Let $e\in B(L^2(\CA,\zeta))$ be the orthogonal projection onto $L^2(M)$. The \textbf{Furstenberg-type entropy} of $(\CA,\zeta)$ with respect to $\varphi$ is defined to be 
$$h_\varphi(\CA,\zeta)=-\varphi(e\log\Delta_\zeta e).$$

\subsection{The standard form of von Neumann algebras.}

Let $\CA$ be a $\sigma$-finite von Neumann algebra. Following \cite{Ha75}, the standard form of $\CA$ is a quadruple $(\CA, H, J, P)$ such that $H$ is a Hilbert space with $\CA\subset B(H)$; $J:H\to H$ is an anti-unitary with $J^2=1$; $P\subset H$ is a self-dual cone, i.e., $P$'s dual cone $P^\mathrm{o}:=\{\xi\in H\mid\langle \xi,\eta\rangle\geq0, \forall \eta\in P\}$ satisfies $P^\mathrm{o}=P$; and $(\CA, H, J, P)$ satisfies the following conditions:
\begin{itemize}
    \item[(1)] $J\CA J=\CA'\cap B(H)$;
    \item[(2)] $JcJ=c^*$ for any $c\in Z(\CA)$, the center of $\CA$;
    \item[(3)] $J\xi=\xi$ for any $\xi\in P$;
    \item[(4)] $a(JaJ)P\subset P$.
\end{itemize}
The standard form is unique up to unitary equivalence: Assume that $(\CA, \tilde{H}, \tilde{J}, \tilde{P})$ is another quadruple satisfying the above conditions. Then there exists a unique unitary $u:H\to\tilde{H}$ such that 
\begin{itemize}
    \item[(a)] $\pi_H(a)=u\pi_{\tilde{H}}(a)u^*$ ($a\in \CA$), where $\pi_H:\CA \to B(H)$ and $\pi_{\tilde{H}}:\CA \to B(\tilde{H})$ are the representations of $\CA$ on $H$ and $\tilde{H}$ respectively;
    \item[(b)] $\tilde{J}=uJu^*$;
    \item[(c)] $\tilde{P}=uP$.
\end{itemize}
Therefore, the standard form of $\CA$ only depends on $\CA$ itself, and we denote it by $(\CA,L^2(\CA),J_\CA,L^2(\CA)_+)$. In particular, for any normal faithful state $\varphi$ on $\CA$, we can take
$$(\CA,L^2(\CA),J_\CA,L^2(\CA)_+)=(\CA,L^2(\CA,\varphi),J_\varphi,L^2(\CA,\varphi)_+),$$
where $J_\varphi=\Delta_\varphi^{1/2}S_\varphi=S_\varphi\Delta_\varphi^{-1/2}$ is the modular conjugation operator of $(\CA,\varphi)$ and $L^2(\CA,\varphi)_+$ is the closure of $\{a(J_\varphi aJ_\varphi)\xi_\varphi \mid a\in \CA \}$ \cite[Lemma 2.9]{Ha75}. 

The Hilbert space $L^2(\CA)$ admits a natural $\CA$-bimodule structure: $a\xi b=a(J_\CA b^* J_\CA)\xi$ ($\xi\in L^2(\CA)$, $a,b\in \CA$). Following \cite[Lemma 2.6]{Ha75}, for a orthogonal projection $p\in\CA$, we have $L^2(p\CA p)=pL^2(\CA)p$ as $p\CA p$-bimodules.

Following \cite[Lemma 2.10]{Ha75}, any normal positive functional $\omega\in\CA_*^+$ admits a unique cyclic vector $\xi_\omega\in L^2(\CA)_+$, i.e., $\omega=\langle \,\cdot\, \xi_\omega,\xi_\omega\rangle$. Moreover, $\xi  \mapsto \omega_\xi=\langle \,\cdot\, \xi,\xi\rangle $ is a homeomorphism between $L^2(\CA)_+$ and $\CA_*^+$ which satisfies that for any $\xi,\eta\in L^2(\CA)_+$,
$$\Vert \xi-\eta \Vert^2\leq \Vert \omega_\xi-\omega_\eta \Vert\leq \Vert \xi-\eta\Vert\Vert \xi+\eta \Vert . $$

\subsection{Connes fusion tensor product.}

For convenience, the inner product in this subsection is linear in the second slot and anti-linear in the first one. Following \cite{Co80} (see also \cite{AP17}), for a tracial von Neumann algebra $(M,\tau)$ and a von Neumann algebra $N$, let $_N H_M$ be a $N$-$M$-bimodule and $_M K_N$ be a $M$-$N$-bimodule. Let
$$H^0=\{\xi\in H\mid \exists c\geq 0 \mbox{ such that } \Vert\xi x\Vert\leq c\Vert x \Vert_{2,\tau} \mbox{ for any } x \in M\}$$
be the set of \textbf{left $M$-bounded vectors} in $H$. For $\xi\in H^0$, let $L_\xi: L^2(M)\to H$ be the bounded operator extended by $\hat{x}\mapsto \xi x \ (x\in M)$. For $\xi_1,\xi_2\in H^0$, let $\langle \xi_1,\xi_2\rangle_M=L_{\xi_1}^*L_{\xi_2}\in M$. Similarly, we can define the set of \textbf{right $M$-bounded vectors} in $K$ and denote it by $^0K$. And for $\eta\in {^0K}$ define $R_\eta:L^2(M)\to K$ by $R_\eta(\hat{x})=x\eta \ (x\in M)$. For $\eta_1,\eta_2\in {^0K}$, let $_M\langle\eta_1,\eta_2\rangle=R^*_{\eta_1}R_{\eta_2}\in M$. 

Then \textbf{Connes tensor fusion product} $H\otimes_M K$ is the Hilbert space deduced from the algebraic tensor product $H^0\odot {^0K}$ by separation and completion relative to sesquilinear form
$$\langle \xi_1\otimes\eta_1,\xi_2\otimes\eta_2\rangle=\langle\eta_1,\langle\xi_1,\xi_2\rangle_M\eta_2\rangle_K,$$
or equivalently,
$$\langle \xi_1\otimes\eta_1,\xi_2\otimes\eta_2\rangle=\langle \xi_1{_M\langle\eta_1,\eta_2\rangle},\xi_2\rangle_H.$$
And $H\otimes_M K$ admits a natural $N$-bimodule structure given by
$$y(\xi\otimes_M \eta)=(y\xi)\otimes_M \eta,\ (\xi\otimes_M \eta)y=\xi\otimes_M (\eta y) \ (y\in N, \xi\in H^0, \eta\in{^0K}). $$

\subsection{Jones' basic construction.}

Following \cite{Jo83} (see also \cite{AP17}), let $( M,\tau)$ be a tracial von Neumann algebra and $B\subset M$ be a von Neumann subalgebra. Let $e_B\in B(L^2( M))$ be the orthogonal projection onto $L^2(B)$. The von Neumann algebra $\langle M, e_B \rangle\subset B(L^2(M))$ generated by $M$ and $e_B$ is called the \textbf{Jones' basic construction} of $B\subset M$, which satisfies
\begin{itemize}
\item[(1)] $e_Bxe_B=E_B(x)e_B=e_BE_B(x)$ for every $x\in M$;
\item[(2)] $J_Me_B=e_B J_M$ for the modular conjugation operator $J_M$ on $L^2(M)$;
\item[(3)] $\langle M, e_B \rangle=(J_M BJ_M)'$;
\item[(4)] $\langle M, e_B \rangle=\overline{\mathrm{span}\{xe_By\mid x,y\in M\}}^\mathrm{w.o.}$;
\item[(5)] $\hat{\tau}(xe_By)=\tau(xy)$ $(x,y\in M)$ defines a normal faithful semi-finite trace $\hat{\tau}$ on $\langle M, e_B \rangle$.
\end{itemize}

The following lemma is a well-known result regarding Jones' basic construction.

\begin{lemma}\label{bimodiso}
Let $( M,\tau)$ be a tracial von Neumann algebra and $B\subset M$ be a von Neumann subalgebra. Let $e_B\in B(L^2( M))$ be the orthogonal projection onto $L^2(B)$ and $\CB=\langle  M,e_B\rangle$. Then there exists an isomorphism between the following $\CB$-bimodules:
$$_\CB L^2(\CB)_\CB \cong {_\CB L^2( M)_B} \bigotimes\limits_{B} {_B L^2( M)_\CB}.$$
Note that $\CB=\langle  M,e_B\rangle=(J_M B J_M)'\subset B(L^2(M))$. Hence $L^2( M)$ can be naturally considered as a $\CB$-$B$ bimodule. And by considering the right action of $\CB$ on $L^2( M)$ given by $J_ M \langle  M,e_B\rangle J_ M$, $L^2( M)$ can also be considered as a $B$-$\CB$ bimodule.
\end{lemma}

\begin{proof}
For convenience, any inner product in this proof is linear in the second slot and anti-linear in the first one.

Note that there exists a normal faithful semi-finite trace $\hat{\tau}$ on $\CB$ satisfying
\begin{equation}\label{hattau}
\hat{\tau}(xe_By)=\tau(xy) \ (x,y\in M).
\end{equation}
Following \cite[Theorem 2.3]{Ha75}, we may assume $L^2(\CB)=L^2(\CB,\hat{\tau})$. Let $\CB_0=\{z\in \CB \mid \hat{\tau}(z^*z)<+\infty\}$. Denote by $\hat{z}\in L^2(\CB)$ the image of $z\in \CB_0$ in the semi-cyclic representation of $\hat{\tau}$. Let $J_\CB $ be the conjugate operator of $\CB$. Since $\hat{\tau}$ is a semi-finite trace, for $z\in \CB_0$ and $a,b\in \CB$, we have 
$$a(J_\CB b^*J_\CB )\hat{z}=\widehat{azb}.$$

Let $\xi \in L^2( M)$ be the $\tau$-cyclic vector and $E_B: M \to B$ the $\tau$-preserving conditional expectation. Define $\Phi_0$: $ M\xi\odot  M\xi\to L^2(\CB)$ by 
$$\Phi_0(a\xi\odot b\xi)=\widehat{ae_Bb}\ (a,b\in  M).$$ 
For any $a_i,b_i\in  M$ $(i=1,2)$, we have
$$\langle b_1\xi, \langle a_1\xi,a_2\xi\rangle_B b_2\xi \rangle=\langle b_1\xi, E_B(a_1^*a_2) b_2\xi \rangle=\tau(b_1^*E_B(a_1^*a_2) b_2).$$
By (\ref{hattau}) and $E_B(a)e_B=e_Bae_B$ $(a\in  M)$ in $\langle  M,e_B\rangle$, we further have 
\begin{align*}
    \tau(b_1^*E_B(a_1^*a_2) b_2)&=\hat{\tau}(b_1^*E_B(a_1^*a_2)e_Bb_2)\\
    &=\hat{\tau}((a_1e_Bb_1)^*(a_2e_Bb_2))=\langle \Phi_0(a_1\xi\odot b_1\xi), \Phi_0(a_2\xi\odot b_2\xi) \rangle.
\end{align*}
Therefore, we have $\langle b_1\xi, \langle a_1\xi,a_2\xi\rangle_B b_2\xi \rangle=\langle \Phi_0(a_1\xi\odot b_1\xi), \Phi_0(a_2\xi\odot b_2\xi) \rangle$ and $\Phi_0$ induces an isometry $\Phi: L^2( M)\otimes_B L^2( M) \to L^2(\CB)$ defined by 
$$\Phi(a\xi\otimes b\xi)= \widehat{ae_Bb} \ (a,b\in  M).$$

Since $\mathrm{span}\{ae_Bb\mid a,b\in M\}$ is weakly dense in $\CB=\langle M,e_B\rangle$, the image of $\Phi$ is dense in $L^2(\CB)$. Also, for $\Phi$ being an isometry, we know that $\Phi$ must be an isomorphism between Hilbert spaces.

Obviously, $\Phi$ is $ M$-bimodular. To prove it is $\langle M,e_B\rangle$-bimodular, we only need to prove $\Phi(e_Ba\xi\otimes b\xi)=\widehat{e_Bae_Bb}$ and $\Phi(a\xi\otimes J_ M e_B^*J_ M b\xi)=\widehat{ae_Bbe_B}$.

For $a,b\in  M$, by $e_Ba\xi=E_B(a)\xi$ in $L^2(B)$ and $e_Bae_B=E_B(a)e_B=e_BE_B(a)$ in $\langle M,e_B\rangle$, we have
$$\Phi(e_Ba\xi\otimes b\xi)=\Phi(E_B(a)\xi\otimes b\xi)=\widehat{E_B(a)e_Bb}=\widehat{e_Bae_Bb}.$$
Also note that $J_ M e_B^*J_ M =e_B$ in $B(L^2( M))$, then we have
$$\Phi(a\xi\otimes J_ M e_B^*J_ M b\xi)=\Phi(a\xi\otimes E_B(b)\xi)=\widehat{ae_B E(b)}=\widehat{ae_Bbe_B}.$$

Therefore, $\Phi$ is a $\langle M,e_B\rangle$-bimodular isomorphism.
\end{proof}

\subsection{Weakly mixing bimodules and property (T).}\label{def (T)} 

Let $M$ be a von Neumann algebra and $H$ be a $M$-bimodule. The \textbf{contragredient $M$-bimodule of $H$} is a $M$-bimodule $\bar{H}$ satisfying that $\bar{H}=\{\bar{\xi}\mid \xi\in H\}$ is the conjugate Hilbert space of $H$ and the $M$-bimodule structure is given by $x\bar{\xi}y=\overline{y^*\xi x^*}$ ($\bar{\xi}\in\bar{H}$, $x,y\in M$). Following \cite{PS12} and \cite{Bo14}, the $M$-bimodule $H$ is said to be \textbf{left weakly mixing} if it satisfies the following equivalent conditions:
\begin{itemize}
    \item [(1)] The $M$-bimodule $H\otimes_M \bar{H}$ contains no non-zero $M$-central vector;
    \item [(2)] There exists a sequence $(u_n)\subset\mathcal{U}(M)$ such that for any $\xi,\eta\in H$,
    $$\lim_n \sup_{y\in (M)_1} |\langle u_n\xi y,\eta\rangle|=0.$$
\end{itemize}
Let $H$ be a left weakly mixing $M$-bimodule and $(u_n)$ be as in condition (2). Then for any $\xi,\eta\in H$, we have $\lim_n  |\langle u_n\xi u_n^*,\eta\rangle|=0$. Therefore, $H$ contains no non-zero $M$-central vector.

Similarly, we can define the \textbf{right weakly mixing} $M$-bimodules. We say that $H$ is \textbf{weakly mixing} if it is both left and right weakly mixing. Following \cite[Proportion 2.4]{PS12}, for any weakly mixing $M$-bimodule $H$ and $n\geq 1$, $H^{\otimes_M^n}$ is still weakly mixing. 

The notion of property (T) is firstly introduced for locally compact groups in \cite{Ka67} and then extended to tracial von Neumann algebras in \cite{CJ85}. Let $M$ be a $\mathrm{II}_1$ factor and $H$ be a $M$-bimodule. A vector $\xi\in H$ is \textbf{$M$-central} if $x\xi=\xi x$ for any $x\in M$. A sequence $(\xi_n)\subset H$ is \textbf{$M$-almost central} if $\lim_n \Vert x\xi-\xi x \Vert=0$ for any $x\in M$. Then $M$ has \textbf{property (T)} if any $M$-bimodule admitting almost central unit vectors contains a non-zero central vector. 

Recently, Tan \cite{Ta23} proved that a separable $\mathrm{II}_1$ factor does not have property (T) if and only if it admits a weakly mixing bimodule which has almost central unit vectors.

\subsection{Mixing bimodules and Haagerup property.}

Following \cite{PS12}, for a von Neumann algebra $M$, a $M$-bimodule $H$ is said to be \textbf{left mixing} if for every sequence $(u_n)\subset\mathcal{U}(M)$ such that $u_n\to 0$ weakly, we have
$$\lim_n \sup_{y\in (M)_1} |\langle u_n\xi y,\eta\rangle|=0.$$

Similarly, we can define the \textbf{right mixing} $M$-bimodules. We said that $H$ is \textbf{mixing} if it is both left and right mixing.

For the same reason in the weakly mixing case, any mixing bimodule has no non-zero central vector. Let $H$ be a left mixing $M$-bimodule and $K$ be a $M$-bimodule. Then for any non-zero vectors $\xi_1,\xi_2\in H^0$, $\eta_1,\eta_2\in {^0K}$ and $(u_n)\subset\mathcal{U}(M)$ such that $u_n\to 0$ weakly, we have
$$\begin{aligned}
&\lim_n \sup_{y\in (M)_1} |\langle u_n(\xi_1\otimes\eta_1) y,\xi_2\otimes\eta_2\rangle|\\
=&\lim_n \sup_{y\in (M)_1} |\langle u_n\xi_1{_M\langle \eta_1y,\eta_2\rangle},\xi_2\rangle|\\
\leq&\lim_n \sup_{z\in (M)_1} \Vert R_{\eta_1}\Vert\Vert R_{\eta_2}\Vert|\langle u_n\xi_1z,\xi_2\rangle|\\
=&0,
\end{aligned}$$ 
where the inequality holds for taking $z=\Vert R_{\eta_1}\Vert^{-1}\Vert R_{\eta_2}\Vert^{-1}{_M\langle \eta_1y,\eta_2\rangle}\in (M)_1$.
Therefore, $H\otimes_M K$ is still left mixing. Similarly, when $H$ is a right mixing $M$-bimodule and $K$ is a $M$-bimodule, $K\otimes_M H$ is still right mixing. In particular, when $H$ is mixing, $H^{\otimes_M^n}$ is still mixing for any $n\geq 1$.

The notion of Haagerup property is firstly introduced for locally compact groups in \cite{Ha79} and then extended to tracial von Neumann algebras in \cite{Ch83}. Following \cite{BF11}, \cite{OOT17} and \cite{DKP22}, a $\mathrm{II}_1$ factor $M$ has the \textbf{Haagerup property} if it admits a mixing $M$-bimodule which has almost central unit vectors.

\subsection{Shlyakhtenko's $M$-valued semicircular system.}
Following \cite{Sh99}, let $(M,\tau)$ be a tracial von Neumann algebra and $(H,J)$ be a symmetric $M$-bimodule, i.e., $J:H\to H$ is an anti-unitary operator satisfying $J^2=1$ and $J(x\xi y)=y^*J(\xi)x^*$ for any $\xi\in H$ and $x,y\in M$. Define the \textbf{full Fock space} of $H$ by
$$\mathcal{F}_M(H)=L^2(M)\oplus \bigoplus_{n=1}^{\infty}H^{\otimes^n_M}.$$
Then \textbf{Shlyakhtenko's $M$-valued semicircular system} of $H$ is a tracial von Neumann algebra $(\tilde{M},\tilde{\tau})$ such that $M\subset\tilde{M}\subset B(\mathcal{F}_M(H))$ and $L^2(\tilde{M},\tilde{\tau})=\mathcal{F}_M(H)$ as $M$-bimodules with $\hat{1}\in L^2(M)\subset\mathcal{F}_M(H)$ as the $\tilde{\tau}$-cyclic vector.

The condition of $H$ being symmetric is the key for $\tilde{M}$ to be tracial.
For simplicity, we omit details of the construction of $(\tilde{M},\tau)$ here and refer to \cite{Sh99} for interested readers. The only result that will be used in this paper is the fact that $L^2(\tilde{M},\tilde{\tau})=\mathcal{F}_M(H)$ as $M$-bimodules.

\section{Hyperstates, hypertraces and noncommutative boundary map}

In this section, we will prove several lemmas regarding hyperstates and hypertrace, preparing for the proofs of main theorems. We also develop the noncommutative Furstenberg boundary map (Theorem \ref{NC boundary map}), as another application of Lemma \ref{hypertrace conditional expectation}.

The following lemma explains the relationship between hypertraces and conditional expectations.
\begin{lemma}\label{hypertrace conditional expectation}
Let $(M,\tau)$ be a tracial von Neumann algebra and $\CA$ be a von Neumann algebra such that $M\subset \CA$. Then a hyperstate $\varphi\in \CS_\tau(\CA)$ is a (normal) hypertrace if and only if the associated u.c.p. map $\CP_\varphi: \CA\to B(L^2(M))$ is a (normal) conditional expectation from $\CA$ to $M$.
\end{lemma}

\begin{proof}
Fix a hyperstate $\varphi\in \CS_\tau(\CA)$. Assume that $\varphi$ is a hypertrace. Since we already have $\CP_\varphi(x)=x$ for any $x\in M$, we only need to prove that $\CP_\varphi(T)\in M$ for any $T\in \CA$. Fix a $T\in \CA$. For any $u\in \mathcal{U}(M)$ and $a,b\in M$, we have
$$\langle\CP_\varphi(T)\hat{a},\hat{b}\rangle=\varphi(b^*Ta)=\varphi(u^*b^*Tau)=\langle\CP_\varphi(T)\widehat{au},\widehat{bu}\rangle=\langle(JuJ)\CP_\varphi(T)(Ju^*J)\hat{a},\hat{b}\rangle.$$
Therefore, $\CP_\varphi(T)=(JuJ)\CP_\varphi(T)(Ju^*J)$ for any $u\in \mathcal{U}(M)$. Hence $\CP_\varphi(T)\in M$ for any $T\in \CA$ and $\CP_\varphi: \CA\to M$ is a conditional expectation.

Assume that $\CP_\varphi$ is a conditional expectation from $\CA$ to $M$. Then for any $T\in \CA$, we have $\CP_\varphi(T)\in M$. For any $u\in \mathcal{U}(M)$ and $T\in\CA$, we have 
$$\varphi(u^*Tu)=\langle\CP_\varphi(T)\hat{u},\hat{u}\rangle=\langle(JuJ)\CP_\varphi(T)(Ju^*J)\hat{1},\hat{1}\rangle=\langle\CP_\varphi(T)\hat{1},\hat{1}\rangle=\varphi(T).$$
Hence $u^*\varphi u=\varphi$ for any $u\in \mathcal{U}(M)$ and $\varphi$ is a hypertrace.

Moreover, by \cite[Proposition 2.1]{DP20}, the hyperstate $\varphi\in \CS_\tau(\CA)$ is normal if and only if $\CP_\varphi$ is normal. Therefore, $\varphi$ is a normal hypertrace if and only if $\CP_\varphi: \CA\to M$ is a normal conditional expectation.
\end{proof}

Before applying Lemma \ref{hypertrace conditional expectation} to the proof of main theorems, we would like to present another application of Lemma \ref{hypertrace conditional expectation}, the noncommutative Furstenberg boundary map.

Let's recall the classical Furstenberg boundary map\cite{Fu63,Fu63a}: Let $\Gamma$ be a countable discrete group, $\mu\in\mathrm{Prob}(\Gamma)$ be a generating measure and $(B,\nu_B)$ be the $(\Gamma,\mu)$-Poisson boundary. Then the following facts hold:
\begin{itemize}
    \item [(a)] For any compact metrizable $(\Gamma,\mu)$-space $(X,\nu)$, there exists an (essentially) unique $\Gamma$-equivalent measurable map $\beta_\nu:B\to\mathrm{Prob}(X)$ such that $\int_B \beta_\nu(b) \d \nu_B(b)=\nu$. The map $\beta_\nu:B\to\mathrm{Prob}(X)$ is usually called \textbf{Furstenberg's boundary map};
    \item[(b)] For any compact metrizable $(\Gamma,\mu)$-space $(X,\nu)$, $(X,\nu)$ is a $(\Gamma,\mu)$-boundary (i.e. $\Gamma$-equivalent measurable factor of $(B,\nu_B)$) if and only if for $\nu_B$-a.e. $\!b\in B$, $\beta_\nu(b)=\delta_{\pi(b)}\in\mathrm{Prob}(X)$ is a Dirac mass;
    \item[(c)] The Poisson boundary $(B,\nu_B)$ is trivial (i.e. $(B,\nu_B)=\{\ast\}$) if and only if for any compact metrizable $\Gamma$-space $X$, any $\mu$-stationary Borel probability measure $\nu\in\mathrm{Prob}(X)$ is $\Gamma$-invariant.
\end{itemize}

Furstenberg's boundary map $\beta_\nu:B\to\mathrm{Prob}(X)$ can be equivalently regarded as the $\Gamma$-equivalent u.c.p. map $\hat{\beta}: C(X)\to L^\infty(B)$: $f\mapsto(b\mapsto\beta_b(f))$. And $\beta_b=\beta_\nu(b)$ is a Dirac mass for $\nu_B$-a.e. $\!b\in B$ if and only if $\hat{\beta}$ is a $*$-homomorphism. Inspired by these facts, we are able to develop the noncommutative Furstenberg's boundary map.

For a tracial von Neumann $(M,\tau)$, a hyperstate $\varphi\in \CS_\tau(B(L^2(M)))$ and a C$^*$-algebra $\CA$ with $M\subset\CA$, we denote the set of $\varphi$-stationary hyperstates on $\CA$ by $\CS_\varphi(\CA)$. For any operator system $\mathcal{C}$ with $M\subset \mathcal{C}$, we denote the set of $M$-bimodular u.c.p. maps from $\CA$ to $\mathcal{C}$ by $\UCP_M(\CA,\mathcal{C})$. Following \cite[Definition 3.7]{Zh23}, up to state preserving isomorphisms, a \textbf{$\varphi$-boundary} is a von subalgebra $(\CB_0,\zeta_0)$ of the $\varphi$-Poisson boundary $(\CB_\varphi,\zeta)$ such that $(M,\tau)\subset(\CB_0,\zeta_0)\subset(\CB_\varphi,\zeta)$. 

\begin{theorem}\label{NC boundary map}
Let $(M,\tau)$ be a tracial von Neumann algebra, $\varphi\in \CS_\tau(B(L^2(M)))$ be a normal regular strongly generating hyperstate and $(\CB_\varphi,\zeta)$ be the $\varphi$-Poisson boundary with canonical hyperstate. Then
\begin{itemize}
    \item [(1)] For any $\mathrm{C}^*$-algebra $\CA$ with $M\subset \CA$, there is a bijection between $\UCP_M(\CA,\CB_\varphi)$ and $\CS_\varphi(\CA)$, which is given by $\Psi\mapsto \zeta\circ\Psi$.
    \item[(2)] For a von Neumann algebra $\CB_0$ with $M\subset\CB_0$ and a  normal faithful $\varphi$-stationary hyperstate $\zeta_0\in \CS_\varphi(\CB_0)$, $(\CB_0,\zeta_0)$ is a $\varphi$-boundary if and only if the associated u.c.p. map $\Psi_{\zeta_0}:\CB_0\to\CB_\varphi$ (i.e. $\zeta_0=\zeta\circ\Psi_{\zeta_0}$) is a $*$-homomorphism.
    \item[(3)] $\CB_\varphi$ is trivial (i.e. $\CB_\varphi=M$) if and only if for any $\mathrm{C}^*$-algebra $\CA$ with $M\subset \CA$, any $\varphi$-stationary hyperstate $\psi$ on $\CA$ must be a $(M,\tau)$-hypertrace.
\end{itemize}
\end{theorem}
\begin{proof}
    (1) Following \cite[Proposition 2.1]{DP20}, there is a bijection between $\CS_\tau(\CA)$ and $\UCP_M(\CA,B(L^2(M)))$ given by $\psi\mapsto \CP_\psi$. It is easy to see that $\psi\in \CS_\tau(\CA)$ is $\varphi$-stationary (i.e. $\CP_\varphi\circ\CP_\psi=\CP_\psi$) if and only if $\CP_\psi(\CA)\subset\har(\CP_\varphi)$. Hence this bijection maps $\CS_\varphi(\CA)$ to $\UCP_M(\CA,\har(\CP_\varphi))$. Let $\CP:\CB_\varphi\to\har(\CP_\varphi)$ be the Poisson transform. Then $\psi\mapsto \Psi_\psi:=\CP^{-1}\circ\CP_\psi$ is a bijection between $\CS_\varphi(\CA)$ and $\UCP_M(\CA,\CB_\varphi)$. Moreover, since for any $\psi\in\CS_\varphi(\CA)$,
    $$\zeta\circ\Psi_\psi=(\zeta\circ\CP^{-1})\circ\CP_\psi=\varphi\circ \CP_\psi=\langle \CP_\varphi\circ\CP_\psi(\,\cdot\,)\hat{1},\hat{1}\rangle=\varphi\ast\psi=\psi,$$
    we know that $\Psi\mapsto \zeta\circ\Psi$ is exactly the inverse of $\psi\mapsto \Psi_\psi$. Therefore, $\Psi\mapsto \zeta\circ\Psi$ is a bijection between $\UCP_M(\CA,\CB_\varphi)$ and $\CS_\varphi(\CA)$.

    (2) If $(\CB_0,\zeta_0)$ is a $\varphi$-boundary, let $\Phi:(\CB_0,\zeta_0)\to (\CB_\varphi,\zeta)$ be the state preserving embedding. Then by the uniqueness of the u.c.p. map $\Psi$ in (1), we must have $\Psi_{\zeta_0}=\Phi$. Therefore, $\Psi_{\zeta_0}:\CB_0\to\CB_\varphi$ is a $*$-homomorphism.

    Assume that $\Psi_{\zeta_0}$ is a $*$-homomorphism. Since $\Psi_{\zeta_0}:(\CB_0,\zeta_0)\to(\CB_\varphi,\zeta)$ preserves the normal faithful state $\zeta_0$, by \cite[Proposition 2.5.11]{AP17}, $\Psi_{\zeta_0}$ is normal and faithful. And by \cite[Proposition 2.5.12]{AP17}, $\Psi_{\zeta_0}(\CB_0)$ is a von Neumann subalgebra of $\CB_\varphi$. Therefore, $\Psi_{\zeta_0}:(\CB_0,\zeta_0)\to(\CB_\varphi,\zeta)$ is a state preserving normal embedding. Hence $(\CB_0,\zeta_0)$ is a $\varphi$-boundary by \cite[Definition 3.7]{Zh23}. 

    (3) $\Rightarrow$: Let $\CA$ be a C$^*$-algebra with $M\subset\CA$ and $\psi\in \CS_\tau(\CA)$ be a $\varphi$-stationary hyperstate. Then by the proof of (1), we must have $\CP_\psi\in \UCP_M(\CA,\har(\CP_\varphi))=\UCP_M(\CA,M)$. Therefore, $\CP_\psi$ is a conditional expectation onto $M$. By Lemma \ref{hypertrace conditional expectation}, we know that $\psi$ must be a $(M,\tau)$-hypertrace.

    $\Leftarrow$: Take $\CA=\CB_\varphi$. Since $\zeta\in\CS_\tau(\CB_\varphi)$ is $\varphi$-stationary, it must be a $(M,\tau)$-hypertrace. By Lemma \ref{hypertrace conditional expectation}, the u.c.p. map $\CP_\zeta:\CB_\varphi\to B(L^2(M))$ is a condition expectation onto $M$. 
    
    By the proof of (1), $\Psi_\zeta=\CP^{-1}\circ\CP_\zeta$ satisfies $\zeta=\zeta\circ\Psi_\zeta=\zeta\circ\mathrm{id}_{\CB_\varphi}$. By the bijection in (1), we must have $\CP^{-1}\circ\CP_\zeta=\mathrm{id}_{\CB_\varphi}$. Hence $\CP_\zeta=\CP:\CB_\varphi\to \har(\CP_\varphi)$, which is an isomorphism between operator systems. Therefore, we must have $\har(\CP_\varphi)=M$ and $B_\varphi$ is trivial. 
\end{proof}

The following lemma shows the existence of normal faithful hyperstate for the inclusion $M\subset \CA$, which we will use later.
\begin{lemma}\label{faithful}
Let $(M,\tau)$ be a separable $\mathrm{II}_1$ factor and $H$ be a separable Hilbert space such that $M\subset B(H)$. Then there exists a normal faithful hyperstate on $B(H)$.
\end{lemma}

\begin{proof}
Following \cite[Theorem 7.1.12]{KR86}, there exists a sequence $(\xi_k)\subset H$ such that $\psi=\sum_{k=1}^{\infty}\langle\,\cdot\,\xi_k,\xi_k\rangle$ defines a normal state on $B(H)$ and satisfies $\psi|_M=\tau$.

Let $\Gamma_1=\{u_n\}\subset \mathcal{U}(M)$ be a countable strongly dense subgroup of $\mathcal{U}(M)$ and $\Gamma_2=\{v_m\}\subset \mathcal{U}(M')$ be a countable strongly dense subgroup of $\mathcal{U}(M')$. Let $\langle M, M'\rangle$ be the von Neumann algebra generated by $M$ and $M'$. Since $M$ is a factor, we have 
$$\langle M, M'\rangle'=M'\cap (M')'=M'\cap M=\C1=B(H)'.$$ 
Therefore, $\langle M, M'\rangle=B(H)$. 

Let $\mathrm{Vect}(\{v_mu_n\})$ be the vector space generated by $\{v_mu_n\mid m, n\geq1\}$ and $\overline{\mathrm{Vect}(\{v_mu_n\})}^\mathrm{s.o.}$ be its strong operator closure. Since $\Gamma_1$ and $\Gamma_2$ are countable strongly dense subgroups of $\mathcal{U}(M)$ and $\mathcal{U}(M')$ respectively, $\mathrm{Vect}(\{v_mu_n\})$ is a $*$-subalgebra of $B(H)$ and $\overline{\mathrm{Vect}(\{v_mu_n\})}^\mathrm{s.o.}$ is a von Neumann subalgebra that contains $M$ and $M'$. Hence 
$$B(H)=\langle M, M'\rangle\subset\overline{\mathrm{Vect}(\{v_mu_n\})}^\mathrm{s.o.}\subset B(H).$$
Therefore, $\overline{\mathrm{Vect}(\{v_mu_n\})}^\mathrm{s.o.}=B(H)$.

For $m,n\geq 1$, take $\lambda_{m,n}>0$ such that $\sum_{m,n\geq1}\lambda_{m,n}=1$. Define a normal state $\psi_0$ on $B(H)$ by
\begin{equation}\label{psi0}
\psi_0(T)=\sum_{m,n\geq1}\lambda_{m,n}\psi(u_n^*v_m^*Tv_mu_n)=\sum_{m,n\geq1}\lambda_{m,n}\sum_{k=1}^{\infty}\langle Tv_mu_n\xi_k,v_mu_n\xi_k\rangle   
\end{equation}
for $T\in B(H)$.

For any $x\in M$, since $\{v_m\}\subset M'$ and $\psi|_M=\tau$ is a trace, we have
$$\psi_0(x)=\sum_{m,n\geq1}\lambda_{m,n}\psi(u_n^*xu_n)=\sum_{m,n\geq1}\lambda_{m,n}\tau(u_n^*xu_n)=\tau(x).$$
Hence $\psi_0$ is a hyperstate.

Without loss of generality, we may assume that $\xi_1\not=0$. Since $\overline{\mathrm{Vect}(\{v_mu_n\})}^\mathrm{s.o.}=B(H)$, we have 
$$H=B(H)\xi_1=\overline{\mathrm{Vect}(\{v_mu_n\})}^\mathrm{s.o.}\xi_1\subset\overline{\mathrm{Vect}(\{v_mu_n\})\xi_1}.$$
Therefore, the vector space generated by $\{v_mu_n\xi_k\mid m, n,k\geq1\}$ is dense in $H$. Hence by (\ref{psi0}), we know that $\psi_0$ is faithful. Therefore, $\psi_0$ is a normal faithful hyperstate on $B(H)$.
\end{proof}

The following lemma shows the existence of normal regular strongly generating hyperstate with finite entropy in $B(L^2(M))$, which we will use later.

\begin{lemma}\label{finite entropy}
Let $(M,\tau)$ be a separable tracial von Neumann algebra. Then there exists a normal regular strongly generating hyperstate $\varphi\in\CS_\tau(B(L^2(M)))$ with $H(\varphi)<+\infty$.
\end{lemma}

\begin{proof}
Take $\{u_n\}$ to be a countable weakly dense subset of $\mathcal{U}(M)$. For $n\geq 1$, take $\lambda_n>0$ such that $\sum_{n=1}^{\infty}\lambda_n=1$ and $\sum_{n=1}^{\infty}-\lambda_n \log \lambda_n<+\infty$ (for example, $\lambda_n=2^{-n}$). Then we can define a normal regular strongly generating hyperstate $\varphi\in\CS_\tau(B(L^2(M)))$ by 
$$\varphi(T)=\sum_{n=1}^{\infty}\lambda_n\langle T \hat{u}_n,\hat{u}_n\rangle \ (T\in B(L^2(M))).$$

Let $A_\varphi\in B(L^2(M))$ be the trace class operator associated to $\varphi$. Only need to prove $H(\varphi)=-\mathrm{Tr}(A_\varphi\log A_\varphi)<+\infty$.

Let $P_n\in B(L^2(M))$ be the orthogonal projection from $L^2(M)$ onto $\C \hat{u}_n$. Then for any $T\in B(L^2(M))$, we have
$$\mathrm{Tr}(A_\varphi T)=\varphi(T)=\sum_{n=1}^{\infty}\lambda_n\langle T \hat{u}_n,\hat{u}_n\rangle=\sum_{n=1}^{\infty}\lambda_n \mathrm{Tr}(P_n T)=\mathrm{Tr}\left(\left(\sum_{n=1}^{\infty}\lambda_n P_n\right) T\right).$$
Therefore, we have
\begin{equation}\label{A= sum P}
A_\varphi=\sum_{n=1}^{\infty}\lambda_n P_n.
\end{equation}
Note that $\{\hat{u}_n\}$ are not necessarily pairwise orthogonal. So (\ref{A= sum P}) is not necessarily the spectral decomposition of $A_\varphi$ and can not be used to calculate $-\mathrm{Tr}(A_\varphi\log A_\varphi)$ explicitly. But it can still be applied to prove $-\mathrm{Tr}(A_\varphi\log A_\varphi)<+\infty$.

By (\ref{A= sum P}), we have
$$H(\varphi)=-\mathrm{Tr}(A_\varphi\log A_\varphi)=\sum_{n=1}^{\infty}-\lambda_n\mathrm{Tr}(P_n\log A_\varphi).$$

 By L\''owner-Heinz theorem (\cite{Lo34} and \cite{He51}, see also \cite[Chapter V]{Bh97}), $x\mapsto \log x$ is operator monotone. Hence for any $n\geq 1$, since $A_\varphi\geq \lambda_n P_n$, we have
 $$\begin{aligned}
-\mathrm{Tr}(P_n\log A_\varphi)=&\lim_{\epsilon\to 0+}-\mathrm{Tr}(P_n\log (A_\varphi+\epsilon))\\
\leq &\lim_{\epsilon\to 0+}-\mathrm{Tr}(P_n\log (\lambda_n P_n+\epsilon))\\
=&\lim_{\epsilon\to 0+}-\log (\lambda_n+\epsilon)\\
=&-\log \lambda_n.
 \end{aligned}$$
Therefore, 
$$H(\varphi)=\sum_{n=1}^{\infty}-\lambda_n\mathrm{Tr}(P_n\log A_\varphi)\leq\sum_{n=1}^{\infty}-\lambda_n \log \lambda_n<+\infty.$$
\end{proof}

\section{Proofs of the main results}

In this section, by an inclusion $M\subset\CA$, we mean a normal embedding between separable von Neumann algebras $M$ and $\CA$.

\begin{definition}\label{entropy gap}
Let $(M,\tau)$ be a separable tracial von Neumann algebra. We say that $(M,\tau)$ has \textbf{entropy gap} if for any normal regular strongly generating hyperstate $\varphi\in \CS_\tau(B(L^2(M)))$, there exists $\epsilon=\epsilon(\varphi)>0$ such that for any inclusion $M\subset \CA$ without normal conditional expectation (from $\CA$ to $M$) and any normal faithful hyperstate $\zeta\in \CS_\tau(\CA)$, one has $h_\varphi(\CA,\zeta)>\epsilon$.
\end{definition}

\begin{definition}
Let $(M,\tau)$ be a separable tracial von Neumann algebra. We say that an inclusion $M\subset\CA$ admits \textbf{almost vanishing Furstenberg entropy} if there exists a normal regular strongly generating hyperstate $\varphi\in\CS_\tau(B(L^2(M)))$ and a net (equivalently, a sequence) of normal faithful hyperstates $(\varphi_i)_{i\in I}$ on $\CA$ such that $$\lim_i h_\varphi(\CA,\varphi_i)=0.$$
\end{definition}

For a $\mathrm{II}_1$ factor $(M,\tau)$ with property (T), \cite[Theorem 6.2]{DP20} proves the existence of such an $\epsilon(\varphi)$ for any normal regular strongly generating hyperstate $\varphi\in \CS_\tau(B(L^2(M)))$ with finite sums in its standard form (i.e. $\varphi=\sum_{k=1}^{n}\langle\, \cdot \, \hat{z_k},\hat{z_k}\rangle $). But their proof also applies to general normal regular strongly generating hyperstates (i.e. $\varphi=\sum_{k=1}^{\infty}\langle\, \cdot \, \hat{z_k},\hat{z_k}\rangle $). Hence any $\mathrm{II}_1$ factor with property (T) has entropy gap. In this section, we will prove the inverse: Any separable $\mathrm{II}_1$ factor without property (T) does not have entropy gap.

Now we are ready to prove the main theorem regarding property (T) (Theorem \ref{thm A}), which is the noncommutative analogue of \cite[Corollary 7.5]{AIM21}. 
\begin{theorem}\label{main}
Let $(M,\tau)$ be a separable $\mathrm{II}_1$ factor. The following conditions are equivalent.
\begin{itemize}
\item [(i)] $M$ does not have property (T);
\item [(ii)]There exists an inclusion $M\subset\CA$ such that $L^2(\CA)$ has $M$-almost central unit vectors, but no non-zero $M$-central vectors;
\item [(iii)] There exists an inclusion $M\subset\CA$ such that there exists a $(M,\tau)$-hypertrace on $\CA$, but no normal hypertrace;
\item [(iv)] There exists an inclusion $M\subset\CA$ such that there exists a conditional expectation from $\CA$ onto $M$, but no normal conditional expectation;
\item [(v)] There exists an inclusion $M\subset\CA$ without normal conditional expectation that admits almost vanishing Furstenberg entropy;
\item [(vi)] $M$ does not have entropy gap.
\end{itemize}
\end{theorem}

\begin{proof}
(i) $\Rightarrow$ (ii): Following \cite[Theorem 4.1]{Ta23}, since $M$ does not have property (T), there exists a $M$-bimodule $H$ such that $H$ is (left and right) weakly mixing and has almost central unit vectors $(\xi_n)$. By considering the closed sub-bimodule generated by $(\xi_n)$, we may assume that $H$ is separable. Let $\bar{H}$ be the contragredient $M$-bimodule of $H$. By considering $H\oplus \bar{H}$ with anti-unitary operator $J(x,\bar{y})=(y,\bar{x})$, we may further assume that $H$ is symmetric, which is necessary for the construction of Shlyakhtenko's $M$-valued semicircular system.

Let 
$$\mathcal{F}_M(H)=L^2(M)\oplus \bigoplus_{n=1}^{\infty}H^{\otimes^n_M}$$
be the full Fock space of $H$. Let $(\tilde{M},\tilde{\tau})$ be the $M$-valued semicircular system of $H$. Then we have $L^2(\tilde{M},\tilde{\tau})=\mathcal{F}_M(H)$.

Let $e_M\in B(L^2(\tilde{M}))$ be the orthogonal projection onto $L^2(M)$ and $\CB=\langle \tilde{M},e_M\rangle$. Then by Lemma \ref{bimodiso}, we have
$$_\CB L^2(\CB)_\CB= {_\CB L^2(\tilde{M})}\otimes_M L^2(\tilde{M}) _\CB={_\CB\mathcal{F}_M(H)}\otimes_M\mathcal{F}_M(H)_\CB.$$ 

Let $e_M^\perp=1-e_M$ be the orthogonal projection from $\mathcal{F}_M(H)$ to $\bigoplus_{n=1}^{\infty}H^{\otimes^n_M}$. Let $\CA=e_M^\perp\langle \tilde{M},e_M\rangle e_M^\perp=e_M^\perp\CB e_M^\perp$. Let $\pi:M\to \CA$ be $\pi(x)=e_M^\perp x e_M^\perp\ (x\in M)$. Note that since $e_M^\perp\in M'\cap \CB$, $\pi$ is a normal $*$-homomorphism.

Take $\omega\in\beta \N \setminus \N$. Since $(\xi_n)\subset H\subset \mathcal{F}_M(H)$ are unit vectors, we can define a state $\varphi_0$ on $B(\mathcal{F}_M(H))$ by
$$\varphi_0(T)=\lim_\omega\langle T \xi_n,\xi_n\rangle \ (T\in B(\mathcal{F}_M(H))).$$
Since $(\xi_n)$ are almost central and $M$ is a $\mathrm{II}_1$ factor, we have $\varphi_0|_M=\tau$. Note that $e_M^\perp\xi_n=\xi_n$, we also have that for any $x\in M$,
$$\varphi_0(\pi(x))=\varphi_0(e_M^\perp x e_M^\perp)=\varphi_0(x)=\tau(x).$$
Hence we must have $\varphi_0\circ\pi=\tau$ on $M$ and $\pi$ is an embedding. By identifying $M$ with $e_M^\perp M e_M^\perp$, we now have $M\subset \CA$.

Since $\CA=e_M^\perp \CB e_M^\perp$, by \cite[Lemma 2.6]{Ha75}, we have $L^2(\CA)=e_M^\perp L^2(\CB)e_M^\perp$ as $\CA$-bimodules. Also, by Lemma \ref{bimodiso}, we have
\begin{align*}
L^2(\CA)=e_M^\perp  L^2(\CB) e_M^\perp&=e_M^\perp L^2(\tilde{M})\otimes_M (J_{\tilde{M}} e_M^\perp J_{\tilde{M}})L^2(\tilde{M})\\
&=\left(\bigoplus_{n=1}^{\infty}H^{\otimes^n_M}\right)\otimes_M \left(\bigoplus_{n=1}^{\infty}H^{\otimes^n_M}\right)=\bigoplus_{m,n\geq1}H^{\otimes^{m+n}_M}
\end{align*}
as $\CA$-bimodules. 

Since $H$ is weakly mixing and has almost central unit vectors $(\xi_n)$, we know that $L^2(\CA)=\oplus_{m,n\geq1}H^{\otimes^{m+n}_M}$ has almost central unit vectors but no non-zero central vectors. And since $H$ is separable, we know that $\CA$ is separable.
\\

(ii) $\Rightarrow$ (iii): Let $\CA$ be as in (ii). Take $(\xi_n)$ to be a sequence of almost central unit vectors in $L^2(\CA)$. Take $\omega\in\beta \N \setminus \N$. Define a state $\varphi_0$ on $\CA\subset B(L^2(\CA))$ by
$$\varphi_0(T)=\lim_\omega\langle T \xi_n,\xi_n\rangle \ (T\in \CA).$$
Then $\varphi_0$ is a $(M,\tau)$-hypertrace on $\CA$.

Assume that there exists a normal hypertrace $\psi$ on $\CA$. By \cite[Lemma 2.10]{Ha75}, the associated cyclic vector $\xi_\psi\in L^2(\CA)_+$ of $\psi$ must be $M$-central. However, since $\CA$ satisfies (ii), there does not exist non-zero $M$-central vector in $L^2(\CA)$, contradiction. Therefore, there exists a hypertrace but no normal hypertrace on $\CA$.
\\

(iii) $\Leftrightarrow$ (iv) is a direct corollary of Lemma \ref{hypertrace conditional expectation}. 
\\

(iii) $\Rightarrow$ (v): Let $\CA$ be as in (iii) and assume that $\varphi_0$ is a hypertrace on $\CA$. Then by (iii) $\Leftrightarrow$ (iv) we know that there exists no normal conditional expectation from $\CA$ to $M$.

By Lemma \ref{finite entropy}, take $\varphi\in\CS_\tau(B(L^2(M)))$ to be a normal regular strongly generating hyperstate with $H(\varphi)<+\infty$. Assume that $\varphi=\sum_{k=1}^{\infty}\langle\,\cdot\, \hat{z_k},\hat{z_k}\rangle$ for $\{z_k\}\subset M$. Only need to construct a net of normal faithful hyperstates $(\varphi_i)_{i\in I}$ on $\CA$ with $\lim_i h_\varphi(A,\varphi_i)=0$.

As a direct corollary of \cite[Lemma 10.2.6]{AP17}, the set of normal states on $\CA$ is weak*-dense in the set of states on $\CA$. Let $(\eta_j)_{j\in J}$ be a net of normal states on $\CA$ such that $\varphi_0=\lim_j\eta_j$ with respect to weak* topology. By Hahn-Banach separation theorem, we may further assume that for any $u\in \mathcal{U}(M)$, \begin{equation}\label{u eta u-eta}
    \lim_j\Vert u^*\eta_j u-\eta_j \Vert=0
\end{equation}
Let $\xi_{\eta_j}\in L^2(\CA)_+$ be the $\eta_j$-cyclic vector. Then $u\xi_{\eta_j}u^*$ is the $u^*\eta_j u$-cyclic vector. By \cite[Lemma 2.10]{Ha75}, (\ref{u eta u-eta}) is equivalent to that for any $u\in \mathcal{U}(M)$,
\begin{equation*}
\lim_j\Vert u\xi_{\eta_j} u^*-\xi_{\eta_j} \Vert=0.
\end{equation*}
Hence $(\xi_{\eta_j})$ is a net of $M$-almost central unit vectors in $L^2(\CA)$.

Let $K=L^2(\CA)^{\oplus \infty}$ as a $\CA$-bimodule. Since there exists a net of almost central unit vectors in $L^2(\CA)$, following \cite[Lemma 2.2]{Ta23}, there exists a sequence of almost unit, almost central subtracial vectors $(\xi_n)$ in $K$, where almost unit means that $\lim_n \Vert \xi_n\Vert=1$ and subtracial means that $\langle x\xi_n,\xi_n\rangle\leq \tau(x)$ and $\langle \xi_n x,\xi_n\rangle\leq \tau(x)$ for any $x\in M_+$ and $n\geq 1$ (note that even starting from a net of almost central vectors in $L^2(\CA)$ we can still obtain a sequence of almost central vectors in $K$).

Define a sequence of normal positive functional $(\varphi'_n)$ on $\CA$ by $$\varphi'_n(T)=\langle T \xi_n,\xi_n\rangle_K \ (T\in \CA).$$
Take $\omega\in \beta \N \setminus \N$ and let $\varphi'_0=\lim_\omega\varphi'_n$ with respect to weak* topology. Then $\varphi'_0$ is a $(M,\tau)$-hypertrace on $\CA$.

Also note that for any $n\geq 1$, $\xi_n$ is subtracial. Hence $\tau-\varphi'_n|_M\geq 0$ on $M$. Since $M\subset B(L^2(\CA))$, following \cite[Theorem 7.1.12]{KR86}, there exists a normal positive function $\psi_n$ on $B(L^2(\CA))$ such that $\psi_n|_M=\tau-\varphi'_n|_M$. Then since $(\xi_n)$ is almost unit, we have
$$\Vert \psi_n\Vert=\psi_n(1)=\tau(1)-\varphi'_n(1)=1-\Vert \xi_n\Vert^2\to 0 \ (n\to\infty).$$
Let $\varphi_n=\varphi'_n+\psi_n|_{\CA}$. Then $\varphi_n|_M=\tau$ and $\varphi_n$ is a hyperstate on $\CA$. And we still have $\varphi_n\to \varphi'_0 \  (n\to \omega)$. 

By Lemma \ref{faithful}, there exists a normal faithful hyperstate $\psi_0$ on $B(L^2(\CA))$, after replacing $\varphi_n$ with $(1-2^{-n})\varphi_n+2^{-n}\psi_0|_\CA$, we may assume that each $\varphi_n$ is faithful.

Let $\varphi_n''=(\sum_{k=0}^{\infty}2^{-k-1}\varphi^{*k})\ast\varphi_n$, where $\varphi^{*0}=\langle \,\cdot\, \hat{1},\hat{1}\rangle$. Since $\varphi'_0$ is a hypertrace, by Lemma \ref{hypertrace conditional expectation}, $\CP_{\varphi'_0}$ is a conditional expectation onto $M$. Moreover, since $\CP_\varphi$ fixes elements in $M$, we have $\CP_\varphi\circ\CP_{\varphi'_0}=\CP_{\varphi'_0}$, i.e., $\varphi\ast\varphi'_0=\varphi'_0$. Therefore, by the continuity of convolution \cite[Lemma 2.1]{DP20}, we have the weak* convergence
$$\varphi_n''=\left(\sum_{k=0}^{\infty}2^{-k-1}\varphi^{*k}\right)\ast\varphi_n\to\left(\sum_{k=0}^{\infty}2^{-k-1}\varphi^{*k}\right)\ast\varphi'_0=\varphi'_0 \ \left(n\to\omega\right). $$
We also have
$$
\begin{aligned}
\varphi\ast\varphi_n''&=\varphi\ast\left(\sum_{k=0}^{\infty}2^{-k-1}\varphi^{*k}\right)\ast\varphi_n\\
&=\left(\sum_{k=0}^{\infty}2^{-k-1}\varphi^{*\left(k+1\right)}\right)\ast\varphi_n\\
&=\left(2\sum_{k=0}^{\infty}2^{-k-1}\varphi^{*k}-\varphi^{*0}\right)\ast\varphi_n\\
&=2\varphi_n''-\varphi_n\\
&\leq2\varphi_n''.    
\end{aligned}
$$
Hence after replacing $\varphi_n$ with $\varphi_n''$, we may assume that for any $n\geq 1$, we have $\varphi\ast\varphi_n\leq2\varphi_n$.

Since for any $u\in \mathcal{U}(M)$, $u^*\varphi_n u-\varphi_n\to u^*\varphi_0' u-\varphi_0'=0 (n\to \omega)$ with respect to weak* topology, by Hahn-Banach separation theorem, there exists a net $(\varphi_i)_{i\in I}$ consists of finite convex combinations of $(\varphi_n)$
such that for any $u\in \mathcal{U}(M)$
\begin{equation}\label{uphiu-phi}
    \lim_i\Vert u^*\varphi_i u-\varphi_i \Vert=0.
\end{equation}
Let $\xi_{\varphi_i}\in L^2(\CA)_+$ be the $\varphi_i$-cyclic vector. Then $u\xi_{\varphi_i}u^*\in L^2(\CA)_+$ is the $u^*\varphi_i u$-cyclic vector. By \cite[Lemma 2.10]{Ha75}, (\ref{uphiu-phi}) is equivalent to that for any $u\in \mathcal{U}(M)$,
\begin{equation*}\label{u1u-1}
\lim_i\Vert u\xi_{\varphi_i} u^*-\xi_{\varphi_i} \Vert=0.
\end{equation*}
It is also equivalent to that for any $z\in M$,
\begin{equation}\label{z1-1z}
\lim_i\Vert z\xi_{\varphi_i}-\xi_{\varphi_i}z \Vert=0.
\end{equation}

Also, since each $\varphi_i$ is a finite convex combination of $(\varphi_n)$, it is still a normal faithful hyperstate and satisfies $\varphi\ast\varphi_i\leq2\varphi_i$.

Now let's prove that $\lim_i h_\varphi(A,\varphi_i)=0$. Let $\Delta_{\varphi_i}$ be the modular operator of $(\CA,\varphi_i)$, $A_\varphi\in B(L^2(M))$ be the trace class operator associated to $\varphi$, and $e:L^2(\CA)\to L^2(M)$ be the orthogonal projection. Following the exact same proof as in \cite[Lemma 3.2]{Zh23}, the inequality $\varphi\ast\varphi_i\leq2\varphi_i$ induces the following inequality in $B(L^2(M))$: for $i\in I$ and $t>0$,
\begin{equation}\label{1+t}
(1+t)^{-1} \leq e(\Delta_{\varphi_i}+t)^{-1}e\leq (\frac{1}{2}A_{\varphi}+t)^{-1}.
\end{equation}

Recall that $J_\CA=\Delta_{\varphi_i}^{1/2}S_{\varphi_i}$ \cite[Lemma VI.1.5]{TakII} and $J_\CA$ fixes $\xi_{\varphi_i}\in L^2(\CA)_+$. We also have that for $z\in M$, $$\xi_{\varphi_i}z=J_\CA z^* J_\CA \xi_{\varphi_i}=J_\CA z^*  \xi_{\varphi_i}=\Delta_{\varphi_i}^{1/2}S_{\varphi_i}(z^*  \xi_{\varphi_i})=\Delta_{\varphi_i}^{1/2}z\xi_{\varphi_i}.$$ Now (\ref{z1-1z}) becomes 
\begin{equation}\label{1-D}
\lim_i\Vert (1-\Delta_{\varphi_i}^{1/2})z\xi_{\varphi_i}\Vert=0.
\end{equation}

Inspired by the proof of \cite[Lemma 5.14]{DP20}, since 
$$\log x = \int_{0}^{+\infty}[(1+t)^{-1}-(x+t)^{-1}]\d t \ (x>0),$$ 
we have 
\begin{equation}\label{int varphi}
h_\varphi(A, \varphi_i)=-\varphi(e\log \Delta_{\varphi_i}e)=\int_{0}^{+\infty}\varphi(e[(\Delta_{\varphi_i}+t)^{-1}-(1+t)^{-1}]e)\d t.   
\end{equation}

For $i\in I$, define $F_i:(0,+\infty)\to \mathbb{R}$ by
$$F_i(t)=\varphi(e[(\Delta_{\varphi_i}+t)^{-1}-(1+t)^{-1}]e).$$
Then by (\ref{1+t}), we have $F_i(t)\geq0$ and $h_\varphi(A,\varphi_i)=\int_{0}^{+\infty} F_i(t) \, \d t.$

Since $\varphi=\sum_{k=1}^{\infty}\langle\,\cdot\, \hat{z_k},\hat{z_k}\rangle$, we have
\begin{align*}
F_i(t)&= \sum_{k=1}^{\infty}\langle e[(\Delta_{\varphi_i}+t)^{-1}-(1+t)^{-1}]e \hat{z_k},\hat{z_k}\rangle\\
&=\sum_{k=1}^{\infty}\langle [(\Delta_{\varphi_i}+t)^{-1}-(1+t)^{-1}] z_k \xi_{\varphi_i},z_k \xi_{\varphi_i}\rangle\\
&=\sum_{k=1}^{\infty}\langle (1+t)^{-1}(\Delta_{\varphi_i}+t)^{-1} (1-\Delta_{\varphi_i}) z_k \xi_{\varphi_i},z_k \xi_{\varphi_i}\rangle\\
&=\sum_{k=1}^{\infty}\langle [(1+t)^{-1}(\Delta_{\varphi_i}+t)^{-1} (1+\Delta_{\varphi_i}^{1/2}) (1-\Delta_{\varphi_i}^{1/2})z_k \xi_{\varphi_i},z_k \xi_{\varphi_i}\rangle\\
&\leq \sum_{k=1}^{\infty}\Vert H_t(\Delta_{\varphi_i})\Vert \cdot \Vert(1-\Delta_{\varphi_i}^{1/2})z_k \xi_{\varphi_i}\Vert\cdot\Vert z_k \xi_{\varphi_i}\Vert.
\end{align*}
Here for $t>0$, $H_t(x)=\frac{1+x^{1/2}}{(1+t)(x+t)}$ on $[0,+\infty)$. Let $C(t)=\sup_{x\geq0}H_t(x)$. Then
$\Vert H_t(\Delta_{\varphi_i})\Vert\leq C(t)$ and 
$$F_i(t)\leq \sum_{k=1}^{\infty}C(t)\cdot \Vert(1-\Delta_{\varphi_i}^{1/2})z_k \xi_{\varphi_i}\Vert\cdot\Vert z_k \xi_{\varphi_i}\Vert.$$

Let's prove that $F_i$ converges to $0$ uniformly on any compact subset of $(0,+\infty)$. Obviously, $C(t)$ is continuous for $t\in(0,+\infty)$, hence bounded on any compact subset of $(0,+\infty)$. For any compact subset $K\subset(0,+\infty)$, let $C_K=\sup_{t\in K}C(t)<+\infty$. For any $\epsilon>0$, since $\sum_{k=1}^{\infty}\Vert \hat{z_k}\Vert^2=1$, there exists a $N\in \N$ such that $\sum_{k=N+1}^{\infty}\Vert \hat{z_k}\Vert^2\leq (4C_K)^{-1}\epsilon$. Also note that 
$$\Vert(1-\Delta_{\varphi_i}^{1/2})z_k \xi_{\varphi_i}\Vert\leq\Vert z_k \xi_{\varphi_i}\Vert+\Vert \Delta_{\varphi_i}^{1/2}z_k \xi_{\varphi_i}\Vert=2\Vert \hat{z_k}\Vert.$$
Then we have that for $t\in K$,
\begin{align*}
F_i(t)&\leq \sum_{k=1}^{N}C_K\cdot \Vert(1-\Delta_{\varphi_i}^{1/2})z_k \xi_{\varphi_i}\Vert\cdot\Vert z_k \xi_{\varphi_i}\Vert+\sum_{k=N+1}^{\infty}C_K \cdot 2\Vert \hat{z_k}\Vert^2\\
&\leq \sum_{k=1}^{N}C_K\cdot \Vert \hat{z_k}\Vert\cdot\Vert(1-\Delta_{\varphi_i}^{1/2})z_k \xi_{\varphi_i}\Vert+\epsilon/2.
\end{align*}
By (\ref{1-D}), there exists an $i_0\in I$ such that for any $i>i_0$, one has
$$\sum_{k=1}^{N}C_K\cdot \Vert \hat{z_k}\Vert\cdot\Vert(1-\Delta_{\varphi_i}^{1/2})z_k \xi_{\varphi_i}\Vert<\epsilon/2.$$
Hence for any $i>i_0$, one has $F_i(t)<\epsilon$ for any $t\in K$. Therefore, $F_i$ converges to $0$ uniformly on any compact subset of $(0,+\infty)$.

Define $G:(0,+\infty)\to \mathbb{R}$ by
$$G(t)=\varphi(( \frac{1}{2}A_{\varphi}+t)^{-1}-(1+t)^{-1}).$$
Then by (\ref{1+t}), we have $0\leq F_i(t)\leq G(t)$. 

Since
$$
\begin{aligned}
\int_{0}^{+\infty}G(t)\,\d t&=\varphi\left(\int_{0}^{+\infty}[( \frac{1}{2}A_{\varphi}+t)^{-1}-(1+t)^{-1}]\,\d t\right)=-\varphi(\log  (\frac{1}{2}A_\varphi))\\
&=-\varphi(\log A_\varphi)+\log 2=-\mathrm{Tr}(A_\varphi \log A_\varphi)+\log 2=H(\varphi)+\log 2<+\infty,
\end{aligned}$$
we have $G\in L^1((0,+\infty))$.

Hence $F_i$ uniformly converges to $0$ on any compact subset of $(0,+\infty)$ and dominated by $G\in L^1((0,+\infty))$ at the same time. So we have
$$\lim_i h_\varphi(\CA,\varphi_i)=\lim_i\int_0^{+\infty}F_i(t)\d t=\int_0^{+\infty} \lim_i F_i(t)\d t=0.$$

Therefore, $M\subset \CA$ admits almost vanishing Furstenberg entropy.
\\

(v) $\Rightarrow$ (vi) is clear.
\\

(vi) $\Rightarrow$ (i) is \cite[Theorem 6.2]{DP20}.
\end{proof}

Note that even $M=L(\Gamma)$ for some ICC countable discrete group $\Gamma$, the von Neumann algebra $\CA$ we construct above is not necessarily a crossed product $L(\Gamma\curvearrowright X)$ for some nonsingular action $\Gamma \curvearrowright (X,\nu_X)$. Hence the proof for the noncommutative case does not recover a proof for the classical case of groups, i.e., \cite[Corollary 7.5]{AIM21}.

Let $(M,\tau)$ be a tracial von Neumann algebra and $H$ be a $M$-bimodule. Recall that a vector $\xi\in H$ is \textbf{left $M$-tracial} if $\langle x\xi,\xi\rangle=\tau(x)$ for any $x\in M$; \textbf{right $M$-tracial} if $\langle \xi x,\xi\rangle=\tau(x)$ for any $x\in M$; \textbf{$M$-tracial} if it is both left and right $M$-tracial. The following corollary shows that a non-property (T) $\mathrm{II}_1$ factor admits a bimodule that has almost central tracial vectors (instead of just unit or subtracial vectors) but no non-zero central vector.

\begin{corollary}\label{best mod}
Let $(M,\tau)$ be a separable $\mathrm{II}_1$ factor without property (T). Then there exists a $M$-bimodule $K$ such that
\begin{itemize}
    \item [(1)]There exists a separable semi-finite von Neumann algebra $\CA$ such that $M\subset \CA$ and $_M K_M={_M L^2(\CA)_M}$;
    \item[(2)] $K$ is weakly mixing;
    \item[(3)] There exists a sequence of almost central unit vectors $(\xi_n)$ in $K$ such that each $\xi_n$ is (left and right) tracial. 
\end{itemize}
\end{corollary}

\begin{proof}
Let $\CA$ be as in the proof of ``(i) $\Rightarrow$ (ii)'' in Theorem \ref{main} and $K=L^2(\CA)$. Then $K=\oplus_{m,n\geq1}H^{\otimes^{m+n}_M}$ for some weakly mixing bimodule $H$, hence $K$ is weakly mixing. Also note that $\CA=e_M^\perp\langle \tilde{M},e_M\rangle e_M^\perp$. Since $\langle \tilde{M},e_M\rangle$ is the commutant of the $\mathrm{II}_1$ factor $J_{\tilde{M}} M J_{\tilde{M}}$, it's semi-finite and so is $\CA$. Hence $K$ satisfies (1) and (2). 

Note that the $\CA$ we take also satisfies the condition (iii) in Theorem \ref{main}. Following the same discussion in the proof of ``(iii) $\Rightarrow$ (v)'' in Theorem \ref{main}, we can take $(\varphi_i)_{i\in I}$ to be the net of normal hyperstates on $\CA$ such that $\lim_i\Vert u^*\varphi_i u-\varphi_i \Vert=0$ for any $u\in \mathcal{U}(M)$. Then by \cite[Lemma 2.10]{Ha75}, the associated cyclic vectors $(\xi_{\varphi_i})\subset L^2(\CA)_+$ are almost central. Note that $(\varphi_i)$ are hyperstates. For any $i\in I$ and $x\in M$, we have
\begin{equation}\label{left tracial}
 \langle x\xi_{\varphi_i},\xi_{\varphi_i}\rangle=\varphi_i(x)=\tau(x).   
\end{equation}
Hence each $\xi_{\varphi_i}$ is left $M$-tracial. Also note that $J_\CA$ fixes every elements in $L^2(\CA)_+$ \cite[Theorem 1.6]{Ha75}. Hence for any $i\in I$ and $x\in M$, we have
\begin{equation}\label{right tracial}
\langle \xi_{\varphi_i}x,\xi_{\varphi_i}\rangle=\langle (J_\CA x^*J_\CA) \xi_{\varphi_i},\xi_{\varphi_i}\rangle=\langle \xi_{\varphi_i},x^*\xi_{\varphi_i}\rangle=\tau(x).
\end{equation}
Therefore, each $\xi_{\varphi_i}$ is both left and right tracial.

For any $z\in M$, we have
$$\lim_i\Vert z\xi_{\varphi_i}-\xi_{\varphi_i}z \Vert=0.$$
Take $\{z_k\}\subset M$ to be a countable $\Vert\cdot\Vert_{2,\tau}$-dense subset of $M$. Since $\{z_k\}$ is just countable, we can take a countable subsequence $(\xi_n)$ of $(\xi_{\varphi_i})$ such that for each $k$, we have
\begin{equation}\label{z_k}
 \lim_n\Vert z_k\xi_n-\xi_n z_k \Vert=0.   
\end{equation}

Only need to prove that $\lim_n\Vert z\xi_n-\xi_n z \Vert=0$ for any $z\in M$. Fix a $z\in M$. For any $\epsilon>0$, since $\{z_k\}\subset M$ is $\Vert\cdot\Vert_{2,\tau}$-dense, there exists a $z_{k_0}\in \{z_k\}$ such that $\Vert z-z_{k_0}\Vert_{2,\tau} <\epsilon/4$. Hence 
$$\begin{aligned}
\Vert z\xi_n-\xi_n z\Vert\leq \Vert z_{k_0}\xi_n-\xi_n z_{k_0}\Vert+ \Vert (z-z_{k_0}) \xi_n\Vert+\Vert \xi_n(z-z_{k_0}) \Vert
\end{aligned}$$
Since $\xi_n$ is tracial, we have
$$\Vert (z-z_{k_0}) \xi_n\Vert=\Vert \xi_n(z-z_{k_0}) \Vert=\Vert z-z_{k_0}\Vert_{2,\tau} <\epsilon/4.$$
Hence
$$\Vert z\xi_n-\xi_n z\Vert\leq \Vert z_{k_0}\xi_n-\xi_n z_{k_0}\Vert+\epsilon/2.$$
By (\ref{z_k}), we further have 
$$\limsup_n\Vert z\xi_n-\xi_n z\Vert\leq \epsilon/2.$$
Hence for any $z\in M$, we have
$$\lim_n\Vert z\xi_n-\xi_n z\Vert=0.$$

Therefore, $(\xi_n)$ is the sequence of almost central tracial vectors that we want and $K=L^2(\CA)$ satisfies (1-3).
\end{proof}

The construction of the von Neumann algebra $\CA$ in Theorem \ref{main} also works for $\mathrm{II}_1$ factor with the Haagerup property. The following theorem (Theorem \ref{thm B}) is the noncommutative analogue of \cite[Corollary 7.6]{AIM21}.

\begin{theorem}\label{(H)}
Let $(M,\tau)$ be a separable $\mathrm{II}_1$ factor. The following conditions are equivalent.
\begin{itemize}
\item [(i)] $M$ has the Haagerup property;
\item [(ii)] There exists an inclusion $M\subset\CA$ such that $L^2(\CA)$ is $M$-mixing and has $M$-almost central unit vectors;
\item [(iii)] There exists an inclusion $M\subset\CA$ such that $L^2(\CA)$ is $M$-mixing and there exists a $(M,\tau)$-hypertrace on $\CA$;
\item [(iv)] There exists an inclusion $M\subset\CA$ such that $L^2(\CA)$ is $M$-mixing and $M\subset \CA$ admits almost vanishing Furstenberg entropy.
\end{itemize}
\end{theorem}

\begin{proof}
(i) $\Rightarrow$ (ii): Since $M$ has the Haagerup property, there exists a $M$-bimodule $H$ that is (left and right) mixing and has almost central unit vectors. By considering the closed sub-bimodule generated by $(\xi_n)$, we may assume that $H$ is separable. By considering $H\oplus \bar{H}$ with anti-unitary operator $J(x,\bar{y})=(y,\bar{x})$, we may further assume that $H$ is symmetric.

Construct $\CA$ as in the proof of ``(i) $\Rightarrow$ (ii)'' in Theorem \ref{main}. Then we have $L^2(\CA)=\oplus_{m,n\geq1}H^{\otimes^{m+n}_M}$ is $M$-mixing and has $M$-almost central unit vectors.
\\

(ii) $\Rightarrow$ (iii): Let $\CA$ be as in the condition (ii). Then the condition that $L^2(\CA)$ is $M$-mixing is already satisfied. And the proof of ``$L^2(\CA)$ has $M$-almost central unit vectors $\Rightarrow$ $\CA$ admits $(M,\tau)$-hypertrace'' is literally the same as the one of ``(ii) $\Rightarrow$ (iii)'' in Theorem \ref{main}.
\\

(ii) $\Rightarrow$ (iii): Let $\CA$ be as in the condition (ii). Then the condition that $L^2(\CA)$ is $M$-mixing is already satisfied. And the proof of ``$\CA$ admits $(M,\tau)$-hypertrace $\Rightarrow$ $M\subset \CA$ admits almost vanishing Furstenberg entropy'' is literally the same as the one of ``(iii) $\Rightarrow$ (v)'' in Theorem \ref{main}.
\\

(iv) $\Rightarrow$ (i): Let $\CA$ be as in the condition (iv). Then there exists a normal regular strongly generating hyperstate $\varphi\in\CS_\tau(B(L^2(M)))$ and a net of normal faithful hyperstates $(\varphi_i)_{i\in I}$ on $\CA$ such that $$\lim_i h_\varphi(\CA,\varphi_i)=0.$$
Assume that $\varphi=\sum_{k=1}^{\infty}\langle\,\cdot\, \hat{z_k},\hat{z_k}\rangle$ for $\{z_k\}\subset M$ such that $\sum_{k=1}^{\infty}z_k^*z_k=\sum_{k=1}^{\infty}z_kz^*_k=1$ and the unital algebra generated by $\{z_k\}$ is weakly dense in $M$.

Inspired by \cite{Ne03}, let $\bar{\varphi}:=\sum_{k=0}^{\infty}2^{-k-1}\varphi^{*k}$, where $\varphi^{*0}=\langle \,\cdot\, \hat{1},\hat{1}\rangle$. Then by \cite[Corollary 5.11]{DP20}, for any normal faithful $\zeta\in\CS_\tau(\CA)$, we still have 
$$h_{\bar{\varphi}}(\CA,\zeta)=\sum_{k=0}^{\infty}2^{-k-1}k \cdot h_{{\varphi}}(\CA,\zeta)=h_{{\varphi}}(\CA,\zeta).$$
Moreover, $\bar{\varphi}=\sum_{k=0}^{\infty}2^{-k-1}\varphi^{*k}$ admits a standard form $\bar{\varphi}=\sum_{j=1}^{\infty}\langle\,\cdot\, \hat{y_j},\hat{y_j}\rangle$, where $$\{y_j\}=\{2^{-1}\}\cup\{2^{-k-1}z_{i_1}z_{i_2}...z_{i_k}\mid k\geq 1, i_l\geq 1 (1\leq l \leq k)\}.$$
Let $\mathrm{Vect}(\{y_j\})$ be the vector space generated by $\{y_j\}$. Since $\varphi$ is strongly generating, $\mathrm{Vect}(\{y_j\})=\C[\{z_k\}]$ is weakly dense in $M$, inducing that it is also strongly dense in $M$ (for a unital subalgebra, self-adjointness of its weak closure is enough to prove the von Neumann bicommutant theorem). Hence after replacing $\varphi$ with $\bar{\varphi}$, we may assume that the vector space generated by $\{z_k\}$ ($\mathrm{Vect}(\{z_k\})$) is $\Vert\cdot\Vert_{2,\tau}$-dense in $M$.

Let $\xi_{\varphi_i}\in L^2(\CA)_+$ be the $\varphi_i$-cyclic vector. Following the same proof in \cite[Theorem 6.2]{DP20}, $\lim_i h_\varphi(\CA,\varphi_i)=0$ induces that $\lim_i\sum_{k=1}^{\infty}\Vert z_k \xi_{\varphi_i}- \xi_{\varphi_i}z_k\Vert=0$. Hence for any $z\in \mathrm{Vect}(\{{z}_k\})$ we have $\lim_i\Vert z \xi_{\varphi_i}- \xi_{\varphi_i}z\Vert=0$. Take a countable $\Vert\cdot\Vert_{2,\tau}$-dense subset $\{w_k\}\subset \mathrm{Vect}(\{{z}_k\})$. Then we can take a subsequence $(\xi_n)$ of $(\xi_{\varphi_i})$ such that for any $k\geq 1$, $\lim_n\Vert w_k \xi_n- \xi_n w_k\Vert=0$. By the same equalities as (\ref{left tracial}) and (\ref{right tracial}), each $\xi_n$ is tracial. Also note that $\{w_k\}$ is $\Vert\cdot\Vert_{2,\tau}$-dense in $M$. Following the same discussion in the proof of Corollary \ref{best mod} we know that $(\xi_n)$ are almost central unit vectors.

Therefore, $L^2(\CA)$ is a mixing $M$-bimodule that has almost central unit vectors. So $M$ has the Haagerup property.
\end{proof}

\section*{Acknowledgement}
This paper was completed under the supervision of Professor Cyril Houdayer. The author would like to thank Professor Cyril Houdayer for numerous insightful discussions and valuable comments on this paper. The author would also like to thank Professor Amine Marrakchi and Hui Tan for useful comments regarding this paper. 





\end{document}